\documentclass[12pt]{article}

\usepackage{vmargin}
\usepackage{fancyhdr}
\usepackage{ifthen}
\usepackage{graphicx} 
\usepackage{indentfirst}
\usepackage[english]{babel}

\usepackage{fouriernc}
\DeclareMathAlphabet{\mathcal}{OMS}{cmsy}{m}{n} 

\usepackage{amsmath}   
\usepackage{amssymb}   
\usepackage{amsfonts}
\usepackage{amsthm}
\usepackage{mathrsfs}

\usepackage{pst-node} 
\usepackage{pstricks}

\usepackage{tweaklist} 
\renewcommand{\itemhook}{\setlength{\topsep}{1pt}%
\setlength{\itemsep}{-1pt}}

\renewcommand{\enumhook}{\setlength{\topsep}{1pt}%
\setlength{\itemsep}{-1pt}}

\usepackage[colorlinks=true,linkcolor=magenta,citecolor=blue]{hyperref}

\usepackage{comment}
\usepackage{color}

\specialcomment{notes}{%
  \begingroup\color{red}\begin{center}\begin{minipage}{12cm} $\blacktriangleright\ $}{%
  \hfill$\ \blacktriangleleft$ \end{minipage} \end{center}\endgroup}

\setpapersize{A4}
\setmarginsrb{33mm}{20mm}{33mm}{25mm}{10mm}{6mm}{0mm}{10mm}
\fancypagestyle{plain}{%
\fancyhf{} 
\lfoot[\thepage]{}
\cfoot[]{}
\rfoot[]{\thepage}
}

\newcommand{\sk}{\smallskip}
\newcommand{\mk}{\medskip}
\newcommand{\bk}{\bigskip}
\renewcommand{\emptyset}{\ensuremath{\varnothing}}     

\newcommand{\G}{\mathbf{G}}
\newcommand{\B}{\mathbf{B}}

\newcommand{\T}{\mathbf{T}}
\renewcommand{\P}{\mathbf{P}}
\renewcommand{\L}{\mathbf{L}}

\newlength{\leftlength}
\newlength{\rightlength}
\newlength{\calculskip}
\newcommand{\calculvskip}[1]{%
  \ifthenelse{#1 = 0}{\setlength{\calculskip}{0pt}}{}%
  \ifthenelse{#1 = 1}{\setlength{\calculskip}{\smallskipamount}}{}%
  \ifthenelse{#1 = 2}{\setlength{\calculskip}{\medskipamount}}{}%
  \ifthenelse{#1 = 3}{\setlength{\calculskip}{\bigskipamount}}{}%
  \ifthenelse{#1 = 4}{\setlength{\calculskip}{1cm}}{}%
  \vskip\calculskip
}

\newcommand{\leftcentersright}[4][2]{%
        \settowidth{\leftlength}{#2}%
        \settowidth{\rightlength}{#4}%
        \calculvskip{#1}
        \noindent#2\hskip-\leftlength%
        \hfill#3\hfill
        \mbox{}\hskip-\rightlength#4%
        \vskip\calculskip%
        }

\newcommand{\centers}[2][2]{\leftcentersright[#1]{}{#2}{}}
\newcommand{\leftcenters}[3][2]{\leftcentersright[#1]{#2}{#3}{}}
\newcommand{\centersright}[3][2]{\leftcentersright[#1]{}{#2}{#3}}

\makeatletter
\def\svhline{%
  \noalign{\ifnum0=`}\fi\hrule \@height2\arrayrulewidth \futurelet
   \reserved@a\@xhline}
   
\def\hlinewd#1{%
\noalign{\ifnum0=`}\fi\hrule \@height #1 %
\futurelet\reserved@a\@xhline}
 \makeatother

\numberwithin{equation}{section}

\newtheorem{thm}[equation]{Theorem}
\newtheorem{lem}[equation]{Lemma}  
\newtheorem{cor}[equation]{Corollary}

\newtheorem{conj}[equation]{Conjecture}
\newtheorem*{thm2}{Theorem}
\newtheorem*{conj2}{Conjecture 1}

\theoremstyle{definition}

\newtheorem{rmk}[equation]{Remark}

\newcommand{\X}{\mathrm{X}}

\newcommand{\Hc}{\mathrm{H}_c}

\newcommand{\Rgc}{\mathrm{R}\Gamma_c}

\newcommand{\F}{\mathbb{F}}
\newcommand{\qlb}{\overline{\mathbb{Q}}_\ell}

\usepackage{pstricks-add}
\usepackage{lscape}
\usepackage{multirow}
\usepackage{rotating}

\begin{document}

\title{Cohomology of Deligne-Lusztig varieties for groups of type $A$}
\author{Olivier Dudas\footnote{Oxford Mathematical Institute.}
\footnote{The author is supported by the EPSRC, Project No EP/H026568/1 and by Magdalen College, Oxford.}}

\maketitle

\begin{abstract} We study the cohomology of parabolic Deligne-Lusztig varieties associated to unipotent blocks of $\mathrm{GL}_n(q)$. We show that the geometric version of Brou\'e's conjecture over $\qlb$, together with Craven's formula, holds for any unipotent block whenever it holds for the principal $\Phi_1$-block.
\end{abstract}

\section*{Introduction}

Let $\G$ be a connected reductive algebraic group over  $\F = \overline{\F}_p$ with an  $\F_q$-structure associated to a Frobenius endomorphism $F$. Let $\ell$ be a  prime number different from $p$ and $b$ be a unipotent $\ell$-block of $\G^F = \G(\F_q)$. When $\ell$ is large, the defect group of $b$ is abelian, and the geometric version of Brou\'e's conjecture predicts that the cohomology of some Deligne-Lusztig variety should induce a derived equivalence between $b$ and its Brauer correspondent \cite{BMa2}. 

\sk

When the centraliser of the defect group of $b$ is a torus, then in \cite{BMi2}  Brou\'e and Michel identified which specific class of Deligne-Lusztig varieties should be considered. They correspond to good $d$-regular elements or equivalently to $d$-roots of $\boldsymbol{\pi} = \mathbf{w}_0^2$ in the Braid monoid. In a recent work \cite{DM3}, Digne and Michel introduced the notion of $d$-\emph{periodic element}  to generalise this to the parabolic setting. If $b$ is a unipotent $\Phi_d$-block, then it is to be expected that there exists a $d$-periodic element $(\mathbf{I},\mathbf{w})$ such that the corresponding parabolic Deligne-Lusztig variety $\widetilde \X(\mathbf{I},\mathbf{w}F)$ is a good candidate for inducing the derived equivalence predicted by Brou\'e's conjecture. Furthermore, Chuang and Rouquier conjectured in \cite{CR} that this equivalence is perverse, with a perversity function that has recently been conjectured by Craven in \cite{Cra}. Surprisingly, it can be expressed by a function $C_d$ depending only on the generic degrees of the corresponding characters.

\sk

If we restrict our attention to the characteristic zero  then we obtain a conjectural explicit description of the unipotent part of the cohomology of $\widetilde \X(\mathbf{I},\mathbf{w}F)$. The fundamental property that we derive from Brou\'e's conjecture is that the cohomology groups of $\X(\mathbf{I},\mathbf{w}F)$ are mutually disjoint. More precisely, it can be formulated as follows:

\begin{conj2}\label{broue}Any $d$-cuspidal pair is conjugate to a pair $(\L_I^{\mathbf{w} F},\chi)$ where $(\mathbf{I} , \mathbf{w})$ is a $d$-periodic element. Moreover, if $\mathcal{F}_\chi$ is the corresponding $\overline{\mathbb{Q}}_\ell$-local system on $\X(\mathbf{I},\mathbf{w}F)$, then $(\mathbf{I},\mathbf{w})$ can be chosen such that

\begin{itemize}

\item[$\mathrm{(i)}$] The $\overline{\mathbb{Q}}_\ell \G^F$-modules $\mathrm{H}^i\big(\X(\mathbf{I},\mathbf{w}F),\mathcal{F}_\chi\big)$ are mutually disjoint. 

\item[$\mathrm{(ii)}$] If $\rho$ is a irreducible unipotent constituent of $\mathrm{H}^i\big(\X(\mathbf{I},\mathbf{w}F),\mathcal{F}_\chi\big)$ then $\rho$ lies in the $\Phi_d$-block associated to $\chi$ and $i = C_d(\deg \rho / \deg \chi)$.

\end{itemize}

\end{conj2}

\noindent In addition, the endomorphism algebra $\mathrm{End}_{\G^F}\big(\mathrm{H}^\bullet(\X(\mathbf{I},\mathbf{w}F),\mathcal{F}_\chi)\big)$ should be endowed with a natural structure of  $d$-cyclotomic Hecke algebra. Let us note the following important consequence of this property: the eigenvalue of any sufficiently divisible power $F^m$ of $F$ on $\rho$ should be $q^{ m \, (a_\rho + A_\rho - a_\chi-A_\chi)/d}$.

\sk

The choice of a specific $d$-periodic element in this conjecture is not very relevant: it is conjectured that any other $d$-periodic element $(\mathbf{I},\mathbf{w}')$ can be obtained from $(\mathbf{I},\mathbf{w}')$ by cyclic shifts, so that the cohomology of the corresponding varieties are isomorphic. This has already been proven when $\mathbf{I} = \emptyset$ and $F$ acts trivially on $W$ (see \cite[Remark 7.4]{DM3}).

\sk

When $d = 1$, the unipotent blocks correspond to the usual Harish-Chandra series. In particular, when $(\G,F)$ has type $A$, there is a unique unipotent block and it contains all the unipotent characters. The  purpose of this paper is to show that from the cohomology of $\X(\boldsymbol \pi)$ one can actually deduce all the other interesting cases (see Corollary \ref{corA}):

\begin{thm2}\label{mainthm}For groups of type $A$, Conjecture \hyperref[broue]{1} holds whenever it holds for $d=1$, that is for $\X(\boldsymbol \pi)$.
\end{thm2}

Let us emphasize that Conjecture \hyperref[broue]{1} is known to be true only in a very small number of cases, namely when $d=h$ is the Coxeter number by Lusztig \cite{Lu}, for groups of rank $2$ by Digne, Michel and Rouquier \cite{DMR} and when $d=n$ for  $A_n$ and $d=4$ for $D_4$ by Digne and Michel \cite{DM2}. Therefore this theorem represents a very important step towards a
proof of the geometric version of Brou\'e's conjecture.

\sk

Even though this result depends on the conjectural description of $\X(\boldsymbol \pi)$, one can give  an effective proof of Conjecture  \hyperref[broue]{1} for principal $\Phi_d$-blocks when $d > (n+1)/2$. In that case the defect group is cyclic, and the modular representation theory of the block is fully understood. We will address this problem in a subsequent paper, where we will  compute the cohomology of $\overline{\mathbb{Z}}_\ell$ of the corresponding Deligne-Lusztig variety.

\sk

To give a flavour of the proof of the main theorem, recall that for groups of type $A$, we have by \cite{BMM} a combinatorial description of the Deligne-Lusztig induction associated to $\widetilde \X(\mathbf{I},\mathbf{w}F)$. In terms of partitions, it corresponds to adding a certain number of $d$-hooks. By transitivity, one can decompose $\widetilde \X(\mathbf{I},\mathbf{w}F)$ by means of simpler varieties $\widetilde \X_{n,d}$, each of  which corresponds to adding a single $d$-hook to a partition. Now, using  the methods developed in \cite{Du5} one can compute the cohomology of (some quotient of) $\widetilde \X_{n,d}$ in terms of  $\widetilde \X_{n-1,d}$ and $\widetilde \X_{n-1,d-1}$ (see Theorem \ref{restrA}), providing an inductive argument to tackle Conjecture \hyperref[broue]{1}.  

\sk

Note finally that Theorem \ref{restrA} can be generalised to many other situations in type $B$, $C$ and $D$. However, two main problems arise: firstly, the limit case is either $\X(w_0)$ or $\X(\boldsymbol \pi)$ and does not contain all the unipotent characters. Secondly, the methods in \cite{Du5} work obviously for non-cuspidal unipotent characters only. We can obtain partial results on the principal series in that situation, which we believe are too coarse to be mentioned in this paper.

\section{Parabolic varieties in type \texorpdfstring{$A$}{A}}

Throughout this paper, $\G$ will denote any connected reductive algebraic group of type $A_n$ over 
$\F = \overline{\F}_p$. We will consider a Frobenius endomorphism $F : \G \longrightarrow \G$ defining a standard $\F_q$-structure on $\G$. Since we will be interested in unipotent characters only, we will not make any specific choice for $(\G,F)$ in its isogeny class. If $\mathbf{H}$ is any $F$-stable subgroup of $\G$ we will denote by $H = \mathbf{H}^F$ the associated finite group. 

\sk

The Weyl group $W$ of $\G$ is the symmetric group $\mathfrak{S}_n$ and its Braid monoid $B^+$ is the usual Artin monoid. It is generated by a set $\mathbf{S} = \{\mathbf{s}_1, \ldots, \mathbf{s}_n\} $ corresponding to simple reflections $s_1, \ldots, s_n$ of $W$. Following \cite{DM3}, we define for $1 \leq d \leq n+1$ 

\centers{$\mathbf{v}_d = \mathbf{s}_{1} \mathbf{s}_2 \cdots \mathbf{s}_{n-\lfloor \frac{d}{2}\rfloor} \mathbf{s}_n \mathbf{s}_{n-1} \cdots \mathbf{s}_{\lfloor \frac{d+1}{2}\rfloor} $}

\leftcenters{and}{$\mathbf{J}_d =  \big\{\mathbf{s}_i \ | \ \lfloor\frac{d+1}{2}\rfloor + 1 \leq i \leq n-\lfloor\frac{d}{2}\rfloor \big\} \, \subset \, \mathbf{S}.$}

\noindent We are interested in computing the cohomology of the variety

\centers{$\mathrm{X}_{n,d} \, = \, \mathrm{X}( \mathbf{J}_d, \mathbf{v}_d F)$}

\noindent with coefficients in any unipotent local system. Note that for $d>1$, the element $\mathbf{v}_d$ is reduced so that we can work with the variety $\mathrm{X}(J_d,v_dF)$. By \cite[Lemma 11.7 and 11.8]{DM3}, the pair $(\mathbf{J}_d,\mathbf{v}_d)$ is \emph{$d$-periodic} so that it makes sense to study the cohomology of $\X_{n,d}$. Recall from \cite{DM3} that a $d$-periodic element is any pair $(\mathbf{I}, \mathbf{b})$ with $\mathbf{I} \subset \mathbf{S}$ and $\mathbf{b} \in B^+$ such that 
 $\mathbf{b}F(\mathbf{b}) \cdots F^{d-1} (\mathbf{b}) = \boldsymbol \pi / \boldsymbol \pi_I$ where $\boldsymbol \pi = \mathbf{w}_0^2$ is the generator of the pure Braid group. 
It has been shown in \cite{DM3} that this forces $\mathbf{b}F$ to normalise $\mathbf{I}$.
Note that  when $d \leq (n+1)/2$, $\mathbf{v}_d$ is not \emph{maximal} in the sense that it is not extendable to a  $d$th root of $\boldsymbol \pi / \boldsymbol \pi_I$ for a  proper subset $\mathbf{I}$ of $\mathbf{J}_d$. However, it can still be used to associate to any unipotent block a "good" parabolic Deligne-Lusztig variety. Before making any precise statement, we shall briefly recall the combinatorial objects that we will use.

\mk 

\noindent \thesection.1. \textbf{$\Phi_d$-blocks of $G$.} The unipotent characters of $G$ are labeled by the partitions of $n+1$. If $\lambda$ is such a partition, we will denote by $\chi_\lambda$ the corresponding character, with the convention that $\chi_{(1,1,\ldots,1)} = \mathrm{St}_G$ is the Steinberg character of $G$. We shall also fix a representation $V_\lambda$ over $\qlb$ of character $\chi_\lambda$. For $1 \leq d \leq n+1$, the pair $(\L_{J_d},v_d F)$ represents a $d$-Levi subgroup of $\G$. From \cite{BMM}, we know how to express the $d$-Harish-Chandra induction in terms of  combinatorics of partitions. To fix the notation,  let $\mu$ be a partition of $n+1-d$ and $X = \{x_1 < x_2 <  \cdots <  x_s\}$ be a $\beta$-set associated to $\mu$. We may and we will assume that $X$ is big enough, so that it contains $\{0,1,\cdots, d-1\}$.  Let $X'$ be the subset of $X$ defined by $X' = \{x \in X \, | \, x+d \notin X\}$. It represents the possible $d$-hooks that can be added to $\mu$. For $x \in X'$ we will denote by  $ \mu * x$ the partition of $n+1$ which has $(X \smallsetminus \{x\}) \cup \{x + d\}$ as a $\beta$-set.

\sk

 We fix an $F$-stable Tits homomorphism $t : B^+ \longrightarrow N_\G(\T)$. By \cite{DM3} the variety $\X_{n,d}$ has an étale covering $\widetilde \X_{n,d} = \widetilde \X(\mathbf{J}_d, \mathbf{v_d} F)$ with Galois group  $\L_{J_d}^{t(\mathbf{v_d})F}$. Since $(\L_{J_d},t(\mathbf{v}_d)F)$ is a split group of type $A_{n-d}$,  the partition $\mu$ defines a unipotent local system $\mathcal{F}_\mu$ on $\X_{n,d}$ such that $\Hc^\bullet(\X_{n,d},\mathcal{F}_\mu)$ and $\Hc^\bullet(\widetilde \X_{n,d},\qlb)_{\chi_\mu}$ are isomorphic. Then we deduce from \cite[Section 3.4]{BMM} that there exist signs $\varepsilon_x=\pm1$ such that the $d$-Harish-Chandra induction of $\chi_\mu$ is given by

\centers{$  \displaystyle \mathrm{R}_{\L_{J_d}}^\G (\chi_\mu) \, = \, \sum (-1)^i \Hc^i(\X_{n,d},\mathcal{F}_\mu) \, = \, \sum_{x \in X'} \varepsilon_x \chi_{\mu * x}.$}

\noindent In particular, the $d$-Harish-Chandra restriction of $\chi_\lambda \in \mathrm{Irr}\, G$ is non-zero until we reach the $d$-core $\nu$ of $\lambda$, which corresponds to a $d$-cuspidal character $\chi_\nu$. The unipotent characters in the $\Phi_d$-block of $G$ containing $\chi_\lambda$ are all the characters that can be obtained by successive $d$-inductions from $\chi_\nu$. They correspond to partitions of $n+1$ that have $\nu$ as a $d$-core. 

\mk

\noindent \thesection.2. \textbf{A parabolic variety associated to a $\Phi_d$-block.} The cohomology of the variety $\X_{n,d}$ induces to a minimal $d$-induction since there is no  $d$-split Levi between $(\L_{{J}_d}, t(\mathbf{v_d})F)$ and $(\G,F)$. By transitivity, one can form a Deligne-Lusztig variety $\X(\mathbf{I},\mathbf{w})$ associated to the $d$-cuspidal character $\chi_\nu$. 
Let $n+1-ad$ be the size of $\nu$ and consider for $i = 1, \ldots ,a$ the pairs $(\mathbf{J}_d^{(i)}, \mathbf{v}_d^{(i)})$ where $(\mathbf{J}_d^{(0)} , \mathbf{v}_d^{(0)}) = (\mathbf{J}_d,\mathbf{v}_d)$ and  
$(\mathbf{J}_d^{(i+1)}, \mathbf{v}_d^{(i+1)})$ is the analogue of the pair $(\mathbf{J}_d,\mathbf{v}_d)$ for the split group $(\L_{J_d^{(i)}}, t(\mathbf{v}_d^{(i)} \cdots \mathbf{v}_d^{(1)} ) F)$ of type $A_{n-id}$. Then one can readily check that the pair $(\mathbf{I},\mathbf{w}) = (\mathbf{J}_d^{(a)},  \mathbf{v}_d^{(a)} \cdots \mathbf{v}_d^{(1)})$ is $d$-periodic. 

\sk

By \cite[Proposition 8.26]{DM3} the associated Deligne-Lusztig variety $\widetilde \X(\mathbf{I},\mathbf{w}F)$ is isomorphic to the following almalgamated product

\centers{$ \widetilde \X(\mathbf{J}_d^{(1)}, \mathbf{v}_d^{(1)} F) \times_{\L_{J_d^{(1)}}^{t(\mathbf{v}_d^{(1)}) F}} \cdots  \times_{\L_{J_d^{(a-1)}}^{t(\mathbf{v}_d^{(a-1)} \cdots \mathbf{v}_d^{(1)}) F}} \widetilde \X_{\L_{\mathbf{J}_d^{(a-1)}}}(\mathbf{J}_d^{(a)}, \mathbf{v}_d^{(a)} t(\mathbf{v}_d^{(a-1)} \cdots \mathbf{v}_d^{(1)}) F).$}

\noindent Now each variety in this decomposition corresponds to a variety $\widetilde \X_{n-id,d}$ for some $i = 0,\ldots, a-1$. Since the cohomology of the latter with coefficients in a unipotent local system depends only on the isogeny class of the group (here, the split type $A_{n-id}$) we obtain
\begin{equation} \label{eqminmax}\Rgc\big(\X(\mathbf{I},  \mathbf{w} F\big), \mathcal{F}_\nu)\, \simeq \, \Rgc(\widetilde \X_{n,d},\qlb) \otimes_{A_{n-d}}  \cdots \otimes_{A_{n-(a-1)d}} \Rgc(\X_{n-(a-1)d,d},\mathcal{F}_\nu).
\end{equation}
\noindent Consequently, if we believe that the vanishing property in Conjecture \hyperref[broue]{1} holds for the cohomology of $\X(\mathbf{I},\mathbf{w})$, then  for any partition $\mu$ of $n+1-d$, the graded $G$-module $\Hc^\bullet(\X_{n,d}, \mathcal{F}_\mu)$ should be multiplicity-free.

\mk

\noindent \thesection.3. \textbf{Craven's formula in type $A$.} Conjecturally, the unipotent character $\chi_\lambda$ is a constituent of only one cohomology group of $\mathrm{R}\Gamma\big(\X(\mathbf{I},  \mathbf{w} F\big), \mathcal{F}_\nu)$. Craven  proposed in \cite{ Cra} a formula which gives the  degree of this cohomology group in terms of $d$ and the generic degree of $\chi_\lambda$ and $\chi_\mu$. 
More precisely, he considers a function $C_d$ on some set of enhanced cyclotomic polynomials and conjectured that
\begin{equation}\label{cravenform}\big\langle \chi_\lambda\, ; \, \mathrm{H}^i(\X(\mathbf{I},\mathbf{w}),\mathcal{F}_\nu) \big\rangle_G \neq 0 \iff i= C_d(\deg\, \chi_\lambda) - C_d(\deg \chi_\nu).\end{equation}
Let us recall the definition of $C_d$: assume that $P \subset \mathbb{Q}[x]$ is a polynomial such that the non-zero roots $z_1, \ldots, z_m$ (written with multiplicity) of $P$ are all roots of unity. Let us denote by $d^\circ(P)$ the degree of $P$ and by $v(P)$ its  valuation, that is the degree of $x^{d^\circ(P)} P(x^{-1})$. Then Craven's function $C_d$ is defined by

\centers{$ C_d(P) = \displaystyle \frac{1}{d} \, \big(d^\circ(P) + v(P)\big)   + \#\big\{ i = 1,\ldots,m \, | \, \mathrm{Arg} \, z_i  < 2\pi/d\big\}  - \frac{1}{2} \#\big\{i =1,\ldots,m \, | \, z_i = 1\big\}. $}

\noindent Here, the argument $\mathrm{Arg}\, z$ of a non-zero complex number $z$ is taken in $ [0;2\pi)$. More  generally, if $\zeta = \exp(2\mathrm{i}\pi k/d) $ is a primitive $d$-root of unity, one can define a function $C_\zeta$ by replacing $d$ by $d/k$ and \ref{cravenform} should hold for $d$th roots of $(\boldsymbol \pi / \boldsymbol \pi_I)^k$. Note also that Craven's function is additive: it satisfies $C_\zeta(PQ) = C_\zeta(P) + C_\zeta(Q)$.

\sk

For groups of type $A$, the degree $\deg\, \chi_\lambda$ of the unipotent character $\chi_\lambda$ is explicitely known (see for example \cite[Section 13]{Car}). It is a polynomial in $q$ of degree $A_\lambda$ and valuation $a_\lambda$ and no factors of the form $(q-1)$ can appear. In particular, Craven's function can be written

\centers{$ C_\zeta(\deg \, \chi_\lambda) = \displaystyle \frac{2\pi}{\mathrm{Arg}\, \zeta} \, \big(a_\lambda + A_\lambda \big)   + \#\big\{ i = 1,\ldots,m \, | \, \mathrm{Arg} \, z_i  < \mathrm{Arg} \, \zeta \big\} $}

\noindent where $z_1,\ldots,z_m$ are the roots with multiplicity of the polynomial $\deg\, \chi_\lambda$.
Note that with this description it is already not obvious that the rational number on the right-hand side of \ref{cravenform} is actually an integer. 

\sk

Since $C_d$ is additive, formula \ref{cravenform}  together with the quasi-isomorphism \ref{eqminmax} suggests that the partition $\nu$ should not be necessarily a $d$-core. In the case of an elementary $d$-induction (when $a=1$) we can write everything  explicitely using \cite[Proposition 9.1]{Cra}; the second equality follows from an easy calculation:

\begin{lem}\label{pervA}Let $\mu$ be a partition with corresponding $\beta$-set $X$ that we assume to be large enough. For $x \in X'$, we have

\centers{$ C_d(\deg\, \chi_{\mu * x}) - C_d(\deg\, \chi_{\mu})\, = \, 2\big(n+1-d-x+\#\{y \in X \, | \, y < x\}\big)+ \#\{y \in X \, | \, x < y < x+d\} $}

\leftcenters{and}{$a_{\mu * x} + A_{\mu * x} - a_\mu - A_\mu \, = \, d(n-d+s-x).$}
\end{lem}

These integers give conjecturally the degree of the cohomology group of $\X_{n,d}$ in which $\chi_{\mu * x}$ will appear, as well as the corresponding eigenvalue of the Frobenius. Since we will  work with the cohomology with compact support, we shall rather work with the integers

\centers{$ \pi_d(X,x) \,  = \, 2\big(n+x - \#\{y \in X \, | \, y < x\}\big) - \#\{y \in X \, | \, x < y < x+d\}$}

\leftcenters{and}{$ \gamma_d(X,x) \, = \, n+1+x-s.$}

\noindent They are readily deduced from the previous ones by taking into account the dimension of $\X_{n,d}$ (which is equal to $\ell(v_d) = 2n+1-d$). Now Conjecture \hyperref[broue]{1}  can be deduced from the following:

\begin{conj}\label{conjA}Let $n \geq 1$ be a positive integer and $1 \leq d \leq n+1$.  Let $\mu$ be a partition of $n-d+1$ and $X$ be its $\beta$-set, assumed to be large enough. Then 

\centersright{$ \Rgc \big(\X_{n,d},\mathcal{F}_\mu\big) \, \simeq \,  \displaystyle \bigoplus_{x \in X'} V_{\mu * x}\, [-\pi_d(X,x)\big] \otimes \qlb   \big(\gamma_d(X,x)\big)$}{}

\noindent as a complex of $G\times \langle F \rangle$-modules.

\end{conj}

The purpose of this paper is to prove that this conjecture holds for any $d$ whenever it holds for $d=1$. As a byproduct, we shall deduce the cohomology of  parabolic Deligne-Lusztig varieties associated to any unipotent block from the knowledge of the cohomology of $\X(\mathbf{w}_0^2)$.

\section{Decomposition of the quotient\label{decA}}

The group $(\L_{J_d}, \dot v_d F)$ is split of type $A_{n-d}$, therefore the unipotent representations of the corresponding finite group are labelled by partitions $\mu$ of $n-d+1$. To such a partition one can associate a unipotent $\qlb$-local system  $\mathcal{F}_\mu$  on $\X_{n,d}$.  From  \cite{BMM} we know that the irreducible constituents of the virtual character $\sum (-1)^i \Hc^i(\X_{n,d},\mathcal{F}_\mu)$ correspond to the  partitions of $n+1$ obtained by adding a $d$-hook to $\mu$. The restriction to $\mathfrak{S}_n$ of the corresponding irreducible representation corresponds to a partition obtained by 

\begin{itemize}

\item either restricting the hook (usually in two different ways),
 
\item or restricting $\mu$.

\end{itemize}

The main result of this section gives a geometric interpretation of this phenomenon.

\begin{thm}\label{restrA}Assume that $d\geq2$. Let $I = \{s_j \, | \, 1 \leq j \leq n-1\}$. Let $\mu$ be a partition of $n-d+1$ and $\{\mu^{(j)}\}$ be the set of partitions of $n-d$ obtained by restricting $\mu$. Then there is a distinguished triangle  in $D^b(\qlb L_I \times \langle F \rangle$-$\mathrm{mod})$

\centers{$ \Rgc(\mathbb{G}_m \times \mathrm{X}_{n-1,d-1}, \qlb \otimes \mathcal{F}_\mu) \longrightarrow \Rgc( \X_{n,d},\mathcal{F}_\mu)^{U_I} \longrightarrow \Rgc\big(\mathrm{X}_{n-1,d},\displaystyle \bigoplus \mathcal{F}_{\mu^{(j)}}\big)[-2] (1) \rightsquigarrow$}
\end{thm}

\begin{rmk}\label{rmkindep}From \cite[Proposition 1.1]{BR2} we can deduce that the cohomology of a Deligne-Lusztig variety with coefficients in a unipotent local system depends only on the type of $(\G,F)$. Therefore there is no ambiguity in the statement of the theorem.
\end{rmk}

We will use the results in \cite{Du5} to compute the quotient of $\widetilde \X(J_d, \dot v_dF)$ by the finite group $U_I$. Recall that $\X_{n,d} = \X(J_d,v_dF)$ can be decomposed into locally closed $P_I$-subvarieties $\X_x$, where $x$ is a $I$-reduced-$J_d$ element of $W$. In our situation, at most two pieces will appear:

\begin{lem} Assume that $2 \leq d \leq n$. The variety $\X_x$ is non-empty if and only if $x$ is one of the two following elements:

\begin{itemize}

\item[(1)] $ x_0 = s_n s_{n-1} \cdots s_1$

\item[(2)] $ x_1 = s_n s_{n-1} \cdots s_{n-\lfloor\frac{d}{2}\rfloor +1}$.

\end{itemize}
\end{lem}

\begin{proof} For simplicity, we shall denote $a = \lfloor\frac{d+1}{2}\rfloor$ and $b = n-\lfloor\frac{d}{2}\rfloor$ so that $w = v_d = s_1 \cdots s_b s_n \cdots s_a$ and $J = J_d= \{a+1,\ldots,b\}$. If $x$ is a $I$-reduced-$J$ element, then $x = s_n s_{n-1} \cdots s_i$ with $b+ 1 \leq i \leq n+1$ or $1 \leq i \leq a$. Recall from \cite{Du5} that the variety $\X_x$ is non-empty if and only if there exists $y = y_1\cdots y_r \in W_J$  and an $x$-distinguished subexpression $\gamma$ of $yw$ such that the products of the elements of $\gamma$ lies in $(W_I)^x$. We first observe that for $i \notin \{1, n+1\}$ we have

\centers{$ (W_I)^x = \langle s_1,\ldots,s_{i-2}, s_i s_{i-1} s_i, s_{i+1}, \ldots,s_n\rangle$.}

\noindent Now, since $x$ is reduced-$J$, the subexpression $\gamma$ is the concatenation of $(y_1,\ldots,y_r)$ and an $xy$-distinguished subexpression $\tilde \gamma$ of $w$. If $i > b+1$ or $2 \leq i \leq a$, the group $W_J$ is included in $(W_I)^x$. Therefore the product of the elements of $\tilde \gamma$ must lie in $(W_I)^x$. We shall distinguish two cases:

\mk

\noindent \textbf{Case (1).} We assume that $ i > b+1$. In that case $x$ commutes with any element of $W_J$, so that $\tilde \gamma$ is an $yx$-distinguished subexpression of $w$.  Then

\begin{itemize}

\item if $x$ is trivial (that is if $i = n+1$), then any $y$-distinguished subexpression of $w$ contains necessarily $s_n$ and hence cannot produce any element of $(W_I)^x = W_I$ ;

\item if $x$ is non-trivial then $i \leq n$, and a subexpression of $w$ lies in $(W_I)^x$ if and only if it does not contain $s_i$ or $s_{i-1}$. However, such a subexpression will never be $yx$-distinguished since for all $v$ in $W_I$ we have $vxs_{i-1} > vx$.

\end{itemize} 

\noindent We deduce that the variety $\X_x$ is empty in this case.

\mk

\noindent \textbf{Case (2).} We assume that $2 \leq i \leq a$. The subexpression $\tilde \gamma$ is  ${}^x y x$-distinguished.  Since $i \leq a$,  we have ${}^x W_J = W_{a, \ldots,b-1} $. For $j<i-1$, we have $xs_j = s_j x$; moreover,  $xs_{i-1}$ is $I$-reduced, so that $\tilde \gamma$ should start with $(s_1,\ldots,s_{i-1})$. In that case, the product of the elements of $\tilde \gamma$ cannot belong to $(W_I)^x$. Indeed, a subexpression of $s_{i-1} s_i \cdots s_b s_n \cdots s_a$ starting with $s_{i-1}$ will never give an element of $(W_I)^x$, the only non-trivial situation being the case $a =i$:

\begin{itemize}

\item with the notation in \cite[Section 2.1.2]{DMR} we have $s_{a-1} \underline{s_{a+1} \cdots s_b s_n \cdots s_a} = \underline{s_{a+1}} $ $ \underline{ \cdots s_b s_n \cdots s_{a+1}} s_{a-1} \underline{s_a}$ and neither $s_{a-1}$ nor $s_{a-1} s_a$ belongs to $(W_I)^x$; 

\item $s_{a-1} s_a \underline{s_{a+1} \cdots s_b s_n \cdots s_a} = (s_{a-1} s_a s_{a-1})\,  s_{a-1} \underline{s_{a+1} \cdots s_b s_n \cdots s_{a}}$ and we are back to the previous case.

\end{itemize}

\noindent This forces the variety $\X_x$ to be empty.
\end{proof}

\begin{proof}[Proof of the Theorem] From the previous lemma we deduce that $ \widetilde \X(J_d,v_dF)$ decomposes as a disjoint union  $ \widetilde \X(J_d,v_dF) \, = \, \widetilde \X_{x_0} \cup \widetilde \X_{x_1}$ with $\widetilde \X_{x_0}$ being open. Using \cite{Du5} we shall now determine the cohomology of the quotient of each of these varieties by $U_I$. Throughout the proof, we will denote $\widetilde I = \{s_2, \ldots, s_n\} \subset S$ the conjugate of $I$ by $w_0$.

\sk

When $x = x_0 = w_I w_0$ and $d >2$ we are in the situation of \cite[Proposition 3.4]{Du5}. Indeed, $v_d = s_1 w'$ with $w' \in W_{2,\ldots,n}$ and $s_1$ commutes with $W_{J_d} \subset W_{3,\ldots, n-1}$ so that we obtain

 \centers{$\Rgc\big(U_I \backslash \widetilde \X_{x_0} \, /   N , \qlb\big)\, \simeq \, \Rgc\big(\mathbb{G}_m \times \widetilde \X_{\L_{ I}}(K_{x_0},\dot vF) \, /\,   {}^{\dot x_0} N', \qlb\big) $}
 
 \noindent with $v = {}^{x_0} w'$ and $K_{x_0} = I \cap {}^{x_0}\Phi_{J_d} = {}^{x_0} J_d$. For simplicity, we shall rather consider the conjugate by $x_0$ of the right-hand side
 
 \sk
 
  Recall that $N$ and $N'$ are normal subgroup of $\L_{J_d}$ and are both contained in $\T$. In particular, any unipotent character of $\L_{J_d}^{wF}$ (resp. of $\L_{J_d}^{w'F}$) is trivial on $N$ (resp. $N'$). Consequently, for any unipotent character $\chi$ of $\L_{J_d}^{wF}$ we obtain the following quasi-isomorphism of complexes of $L_I \times \langle F \rangle$-modules:
 
  \centers{$\Rgc\big(U_I \backslash \widetilde \X_{x_0}, \qlb\big)_{\chi} \, \simeq \, \Rgc\big(\mathbb{G}_m \times \widetilde \X_{\L_{ I}}({}^{x_0} J_d,\dot vF), \qlb\big)_{{}^{\dot x_0}\chi}. $}
 
 \noindent Finally, we observe that the varieties $\X_{\L_{ I}}({}^{x_0} J_d,\dot vF)$ and $\X_{n-1,d-1}$ have the same cohomology with coefficients in any unipotent local system. Indeed, if we denote $(s_1, \ldots ,s_{n-1})$ by $(t_1,\ldots,t_{n-1})$ if $d$ is odd or by $(t_{n-1},\ldots,t_1)$ if $d$ is even, then we have
 
\centers{$v =  t_{1}  t _2 \cdots  t_{n-1-\lfloor \frac{d-1}{2}\rfloor}  t_{n-1} \dot t_{n-2} \cdots  t_{\lfloor \frac{d}{2}\rfloor} $}
 
 \noindent which corresponds to the element $v_{d-1}$ in the Weyl group $W_{ I} = \langle t_1, \ldots, t_{n-1}\rangle $ of type $A_{n-1}$. 
 
 \sk
 
 When $x= x_0$ and $d=2$, we can write $v_2 = w w'$ with $w= s_n s_{n-1} \cdots s_2$ and $w' =  s_1 s_2 \cdots s_n = s_1 w''$ so that $\X_{n,2} \simeq \X\big(\{\mathbf{s}_2, \ldots ,\mathbf{s}_{n-1}\}, \mathbf{w}\mathbf{w'}F\big)$. Moreover, via this isomorphism we have 
 
 \centers{$ \X_{x_0} \simeq \displaystyle \bigcup_{y \in W} \X_{(x_0,y)}.$}
  
\noindent We claim that $\X_{(x_0,y)}$ is empty unless $y \in W_I w_0 W_{J_2'}$ where $J_2' = J_2^w =  \{s_3,s_4,\ldots,s_n\}$. The piece $\X_{(x_0,y)}$ consists of  pairs $(px_0 \P_{J_2}, p'y\P_{J_2'})$ with $p,p' \in \P_I$ such that $p^{-1} p' \in x_0 \P_{J_2} w \P_{J_2'} y^{-1}$ and $p'^{-1} \, {}^F p \in y \P_{J_2'} w' \P_{J_2} x_0^{-1}$. In particular, if  $\X_{(x_0,y)}$ is non-empty then the double coset $\P_I y  \P_{J_2'}$ has a non-trivial intersection with  $x_0 \B w$. But $x_0 = w_I w_0$ is reduced-$\widetilde I$ so that $\ell(x_0 w ) = \ell(x_0) + \ell(w)$ and $x_0 \B w \subset \P_I w_0 \P_{J_2'}$. This forces $y$ to lie in $ W_I w_0 W_{J_2'}$. Note that $w_0 (J_2') \subset I$ so that the minimal element in this coset is $x_0$ and we have $\X_{x_0} \simeq \X_{(x_0,x_0)}$.

\sk

  Now $s_1$ commutes with $J_2'$ and we can apply \cite[Proposition 3.4]{Du5} to obtain, after conjugation by $x_0$:

\centers{$\Rgc\big(U_I \backslash \widetilde \X_{x_0} \, /   N , \qlb\big)\, \simeq \, \Rgc\big(\mathbb{G}_m \times \widetilde \X_{\L_{ I}}(\mathbf{K}_{x_0},\mathbf{v} \mathbf{v}'F) \, /\,   {}^{\dot x_0} N', \qlb\big). $}
 
\noindent with $K_{x_0} = {}^{x_0} J_2$,  $v = {}^{x_0} w$ and $v' = {}^{x_0} w''$. If we denote $(s_1,\ldots,s_{n-1})$ by $(t_{n-1},\ldots,t_1)$ we obtain  ${}^{x_0} J_2 = \{t_2, \ldots,t_{n-1}\}$  and $\mathbf{vv}' = \mathbf{t}_1 \cdots \mathbf{t}_{n-1} \mathbf{t}_{n-1} \cdots \mathbf{t}_1$ so that the pair $(\mathbf{K}_{x_0},\mathbf{vv}')$ corresponds to $(\mathbf{J}_1, \mathbf{v}_1)$ in  the Weyl group $W_{I}$ of type $A_{n-1}$. As before, $N$ and $N'$ do not play any role if we consider the unipotent part of the previous quasi-isomorphism.

\sk

When $x = x_1$ we use \cite[Proposition 3.2]{Du5}: the conjugate of $v_d$ by $x_1$ is 

\centers{$  v = x_1v_dx_1^{-1} \, = \, s_1 s_2 \cdots s_{n-1-\lfloor \frac{d}{2}\rfloor}  s_{n-1} s_{n-2} \cdots s_{\lfloor \frac{d+1}{2}\rfloor}  $}

\noindent which corresponds exactly to the element $v_d$ in $W_I$. We can therefore identify the cohomology of the varieties $\X_{\L_I}(K_{x_1}, vF)$ and $\X_{n-1,d}$ with coefficients in any unipotent local system (see Remark \ref{rmkindep}). The group $\P_I \cap {}^{x_1} \L_{J_d}$ is a $\dot v F$-stable parabolic subgroup of ${}^{x_1} \L_{J_d}$ and $\L_{K_{x_1}} = \L_I \cap {}^{x_1} \L_{J_d}$ is a stable Levi complement. Therefore it makes sense to consider the Harish-Chandra restriction ${}^*\mathrm{R}_{K}^{J} \, \chi$ of any unipotent character $\chi$ of $\L_{J_d}^{\dot v_d F} \simeq ({}^{x_1} \L_{J_d})^{\dot v F}$ to $\L_{K_{x_1}}^{\dot v F}$. From \cite[Proposition 3.2]{Du5} (see also \cite[Remark 3.12]{Du5}) we deduce the following quasi-isomorphism

\centers{$ \Rgc\big(U_I \backslash \widetilde \X_{x_1},\qlb\big)_{\chi} [2](-1) \, \simeq \,   \Rgc\big( \widetilde \X_{\L_I}(K_{x_1}, \dot vF),\qlb\big)_{{}^*\mathrm{R}_{K}^{J} \, \chi}$.}

\sk

Let $\mu$ be a partition of $n-d+1$. The cohomology of the variety $\X_{n,d}$ with coefficients in the local system $\mathcal{F}_\mu$ is given by 

\centers{$ \Rgc(\X_{n,d}, \mathcal{F}_\mu) \, \simeq \, \Rgc\big(\widetilde \X(\mathbf{J}_d,\mathbf{v}_dF), \qlb\big)_{\chi_\mu}$}

\noindent where $\chi_\mu$ is the unipotent character of $\L_{J_{d}}^{\dot v_dF}$ corresponding to the partition $\mu$. Since $(\L_{K_{x_1}}, \dot v F)$ is a split group of type $A_{n-d-1}$, the Harish-Chandra restriction of $\chi_\mu$ from $\L_{J_{d}}^{\dot wF} \simeq ({}^{x_1} \L_{J_d})^{\dot v F}$ to $\L_{K_{x_1}}^{\dot vF}$ is the sum of the $\chi_{\mu_i}$'s where the $\mu_i$'s are the partitions of $n-d$ obtained by restricting $\mu$. With with description, we get the following isomorphisms in $D^b(\qlb L_I \times \langle F \rangle$-$\mathrm{mod})$

\centers{$ \Rgc\big(U_I \backslash \widetilde \X_{x_0}, \qlb\big)_{\chi_\mu} \, \simeq \, \Rgc\big(\mathbb{G}_m \times  \X_{n-1,d-1}, \qlb \otimes \mathcal{F}_\mu \big)  $}

\leftcenters{and}{$ \Rgc\big(U_I \backslash \widetilde \X_{x_1},\qlb\big)_{\chi_\mu}  \, \simeq\,   \Rgc\big( \X_{n-1,d},\displaystyle \bigoplus \mathcal{F}_{\mu_i}\big)[-2](1).$}

\noindent We conclude using the distinguished triangle associated to the decomposition $\widetilde \X_{n,d} = \widetilde \X_{x_0} \cup \widetilde \X_{x_1}$ in which $\widetilde \X_{x_0}$ is open.
\end{proof}

\section{Cohomology over \texorpdfstring{$\qlb$}{Ql}}

We have just seen how to relate the Harish-Chandra restriction of the cohomology of $\X_{n,d}$ to the cohomology of smaller parabolic Deligne-Lusztig varieties. We shall now explain how this strategy provides an inductive method for a thorough determination of the cohomology of $\X_{n,d}$ with coefficients in any unipotent local system. The main result in this section gives an inductive strategy towards a proof of Conjecture \hyperref[broue]{1}:

\begin{thm}\label{thmA}Let $n \geq 1$ and $2 \leq d \leq n$.  If Conjecture \ref{conjA} holds for $(n,d+1)$,  $(n-1,d-1)$ and $(n-1,d)$ then it holds for $(n,d)$.
\end{thm}

Note that  we already know  from \cite{Lu} that  Conjecture \ref{conjA} holds in the Coxeter case, corresponding to $(n,n+1)$. Therefore $d=1$ is the only limit case. But $\boldsymbol \pi = \mathbf{w}_0^2$ is a maximal $1$-periodic element  in the sense of \cite{DM3} and in that specific case, a general conjecture for the cohomology has been already formulated in \cite{BMi2}: a unipotent character $\chi_\lambda$ can appear in $\Hc^i(\X(\boldsymbol \pi))$ for $i = 4 \nu_\G - 2 A_\lambda$ only, where $\nu_\G$ is the number of positive roots. An important consequence of  Theorem \ref{thmA} is that knowing the cohomology of $\X(\boldsymbol \pi)$ is sufficient for determining all the other interesting cases:

\begin{cor}\label{corA}For groups of type $A$, Conjecture \hyperref[broue]{1} holds for any $d \geq 1$ as soon as it holds for $d=1$.
\end{cor}

\begin{proof} Assume that Conjecture \hyperref[broue]{1}  holds for $d=1$, that is for the variety $\X(\boldsymbol \pi)$. Let $I = J_1 = \{s_2,\ldots,s_n\}$ and $\mathbf{b} = \mathbf{v}_1 = \mathbf{s}_1 \cdots \mathbf{s}_n \mathbf{s}_n \cdots \mathbf{s}_1$. By \cite[Proposition 8.26]{DM3} we have

\centers{$ \Rgc\big(\X(\boldsymbol \pi),\qlb\big) \simeq \Rgc \big(\widetilde \X(\mathbf{I},\mathbf{b}F), \qlb\big) \otimes_{\qlb \L_I^{t(\mathbf{b})F}} \Rgc\big(\X(\boldsymbol \pi_{I}),\qlb\big).$}

\noindent Since the cohomology of $\X(\boldsymbol \pi_I)$ contains all the unipotent characters of $\L_I^{t(\mathbf{b})F}$, we deduce  that for any partition $\mu$ of $n$, the  groups $\Hc^i(\X_{n,1},\mathcal{F}_\mu)$ are submodules of the cohomology groups of $\X(\boldsymbol \pi)$. Consequently, they are disjoint as soon as Conjecture \hyperref[broue]{1} holds for $\X(\boldsymbol \pi)$. Since we have assumed that it holds also for $\X(\boldsymbol \pi_I)$ we have actually 
\begin{equation}\label{eq2}\Hc^i\big(\X(\boldsymbol \pi),\qlb\big) \, \simeq \, \displaystyle \bigoplus_{\mu - n} \Hc^{i-4\nu_{\L_I} + 2 A_\mu}\big (\X_{n,1}, \mathcal{F}_\mu\big)(2\nu_{\L_I} - a_\mu - A_\mu)
\end{equation}
\noindent as a  $G\times \langle F \rangle$-module. Now, the alternating sum of the cohomology groups of $\X_{n,1}$ represents the Deligne-Lusztig induction from $\L_I^{t(\mathbf{b})F} \simeq L_I$ to $G$. Therefore a character $\chi_\lambda$ appear in $\Hc^\bullet(\X_{n,1},\mathcal{F}_\mu)$ if and only if $\mu$ is the restriction of $\lambda$, or equivalently, if $\lambda$ is obtained from $\mu$ by adding a $1$-hook. This, together with \ref{eq2} and Lemma \ref{pervA} proves that Conjecture \ref{conjA} holds for $\X_{n,1}$, and therefore for any variety $\X_{n,d}$ by \ref{thmA}. We use \ref{eqminmax} to conclude.
\end{proof}

\begin{proof}[Proof of the Theorem]   Let $X$ be a $\beta$-set associated the partition $\mu$ of $n-d+1$. We can always assume that it contains $\{0,1,\ldots,d-1\}$. The partitions $\mu^{(j)}$'s of $n-d$ which are obtained by restricting $\mu$ can be associated to the following $\beta$-set: 

\centers{$X^{(j)} = \{x_1^{(j)} < \cdots < x_s^{(j)}\} \quad \text{with} \quad x_i^{(j)} \, = \left\{ \hskip-1.3mm\begin{array}[c]{l} x_j-1 \ \ \text{if } i=j; \\[3pt] x_i \ \  \text{otherwise.} \end{array}\right. $}

Let $I = \{s_1, \ldots,s_{n-1}\}$. By Theorem \ref{restrA}, the Harish-Chandra restriction of the cohomology of $\X_{n,d}$ can be fitted into the following distinguished triangle:

\centers{$ \Rgc(\mathbb{G}_m \times \mathrm{X}_{n-1,d-1}, \qlb \otimes \mathcal{F}_\mu) \longrightarrow \Rgc( \X_{n,d},\mathcal{F}_\mu)^{U_I} \longrightarrow \Rgc\big(\mathrm{X}_{n-1,d},\displaystyle \bigoplus \mathcal{F}_{\mu^{(j)}}\big)[-2] (1) \rightsquigarrow$}

\noindent Now, if we assume that Conjecture \ref{conjA} holds for both $(n-1,d-1)$ and $(n-1,d)$, the  complexes on the left and  right-hand side are completely determined. Let us examine the different eigenvalues of $F$ that can appear:

\begin{itemize}

\item[(a)] on $\mathscr{C} = \Rgc(\mathbb{G}_m \times \mathrm{X}_{n-1,d-1}, \qlb \otimes \mathcal{F}_\mu)$, the eigenvalues of $F$ are $q^{n+x-s}$ and $q^{n+1+x-s}$ with $x \in X$ such that $x+d-1 \notin X$. The character of the corresponding eigenspace is $\chi_\lambda$ where $\lambda$ is the partition of $n$ obtained by adding to $\mu$ a $(d-1)$-hook represented by $x$; 

\item[(b)] on $\mathscr{D}^{(j)} = \Rgc(\mathrm{X}_{n-1,d}, \mathcal{F}_{\mu^{(j)}})[-2](1)$, the eigenvalues of $F$ are $q^{n+1+x-s}$ where $x \in X^{(j)}$ is such that $x+d \notin X^{(j)}$. The character of the corresponding eigen\-space is $\chi_\lambda$  where $\lambda$ is the partition of $n$ obtained by adding to $\mu^{(j)}$ a $d$-hook represented by $x$.

\end{itemize}

 We shall now determine $\Hc^\bullet(\X_{n,d},\mathcal{F}_\mu)^{U_I}$  by studying each eigenspace of $F$ separately. For $x$ a positive integer, we can separate the following cases:

\mk

\noindent \textbf{Case (1).} Assume first that $x \in X$ and $x+d \notin X$. Let $\lambda = \mu * x$ be the partition of $n+1$ obtained by adding to $\mu$ a $d$-hook from $x$. We want to prove that the $q^{n+1- s +x}$-eigenspace of $F$ on $\Hc^\bullet(\X_{n,d},\mathcal{F}_\mu)^{U_I}$ is non-zero in degree $\pi_d(X,x)$ only  and that its character is the Harish-Chandra restriction of $\chi_\lambda$.

\sk

By (a), the $q^{n+1+x-s}$-eigenspace of $F$ on $\mathscr{C}$ will produce non-zero representations in the following two cases:
\begin{itemize}

\item if $x+d-1 \notin X$, then one obtains a character associated to the $\beta$-set $(X \smallsetminus\{x\}) \cup \{x+d-1\}$ and it is concentrated in degree 

\centers{$\begin{array}{r@{\, \ = \, \ }l}2+\pi_{d-1}(X,x) & 2 + 2\big(n-1+x - \#\{y \in X \, | \, y < x\}\big) - \#\{y \in X \, | \, x < y< x+d-1\} \\[4pt] &
2 \big(n+x - \#\{y \in X \, | \, y < x\}\big) - \#\{y \in X \, | \, x < y < x+d-1\} \\[4pt]
 & \pi_d(X,x). \end{array}$}
 
\item if $x+1 \in X$, then the corresponding  $\beta$-set is $(X \smallsetminus\{x+1\}) \cup \{x+d\}$ and the associated character appears in degree $1 + \pi_{d-1}(X,x+1)$ only. But we have

\centers{$\begin{array}{r@{\, \ = \, \ }l} \pi_{d-1}(X,x+1) & 2\big(n+x - \#\{y \in X \, | \, y < x+1\}\big) - \#\{y \in X \, | \, x+1 < y < x+d\} \\[4pt] &
 2\big(n+x - 1- \#\{y \in X \, | \, y < x\}\big) - \big( \#\{y \in X \, | \, x < y < x+d\}-1\big) \\[4pt]
 & \pi_d(X,x) -1. \end{array}$}
 
\end{itemize} 

On the other hand, the $q^{n+1+x-s}$-eigenspace of $F$ on $\mathscr{D}^{(j)}$ is non-zero if and only if $x \in X^{(j)}$ and $x+d \notin X^{(j)}$. This happens if and only if $x$ and $x+d+1$ are different from $x_j$. In that case, the $\beta$-set corresponding to the character of the eigenspace will be $(X^{(j)} \smallsetminus \{x\}) \cup \{x+d\}$. Furthermore, the degree in which this character will appear is $ 2 + \pi_d(X^{(j)},x)$, which is clearly equal to $\pi_d(X,x)$ in that case.

\sk

Now, the $\beta$-set $Y = (X\smallsetminus \{x\})\cup \{x+d\}$ is associated to the partition $\lambda = \mu * x$.  As mentioned in the beginning of Section \ref{decA}, the restriction of $\lambda$ is obtained by restricting the hook (usually in two different ways) or by restricting $\mu$. In the framework of $\beta$-sets, it corresponds to decreasing specific elements of $Y$: 
\begin{itemize}

\item if $x + d-1 \notin X$, one can replace $x+d$ by $x+d-1$ in $Y$ and we obtain the $\beta$-set $(X\smallsetminus \{x\}) \cup \{x+d-1\}$;

\item if $x + 1 \in X$, one can replace $x+1$ by $x$ in $Y$ and we obtain the $\beta$-set $(X\smallsetminus \{x+1\}) \cup \{x+d\}$;

\item if $x_j \in X$ is different from $x$ or from $x+d+1$, and if $x_j -1 \notin X$, then one can replace $x_j$ by $x_j-1$ in $Y$ to obtain the $\beta$-set $(X^{(j)} \smallsetminus \{x\}) \cup \{x+d\}$.
\end{itemize} 

\noindent This shows that the character of the $q^{n+1+x-s}$-eigenspace of $F$ on $\Hc^\bullet(\X_{n,d},\mathcal{F}_\mu)^{U_I}$ is the Harish-Chandra restriction of $\chi_{\mu * x}$. 

\begin{lem} Let $\lambda$ be a partition of $n+1$, with $n \geq 3$ and let $\chi$ be a (non-necessarily irreducible) unipotent character of $G$. Then the Harish-Chandra restriction of $\chi$ and $\chi_\lambda$ are equal if and only if $\chi = \chi_\lambda$.
\end{lem}

\begin{proof}[Proof of the Lemma] Assume that there exists a partition $\nu = \{\nu_1 \leq \nu_2 \leq \cdots \leq \nu_r\}$ of $n+1$ with $\nu_1 \neq 0$ such that the difference between the Harish-Chandra restriction of $\chi_\lambda$ and $\chi_\nu$ is still a unipotent character. This means that in the Young diagram of $\nu$, any box that can be removed can be replaced to form the Young diagram of $\lambda$. If $\nu \neq \lambda$, this is possible only if $\nu_1 = \nu_2 = \cdots = \nu_r$.
\sk

Let $\chi$ be a unipotent character of $G$ which has the same Harish-Chandra restriction as $\chi_\lambda$. If $\chi \neq \chi_\lambda$, we deduce from the previous argument that all the irreducible constituents of $\chi$ are of the form $\chi_\nu$ with $\nu = (a,a, \ldots,a)$. This can happen if and only if $n=2$ and $\lambda = (1,2)$. \end{proof}

When $n \geq 3$, we deduce that the $q^{n+1+x-s}$-eigenspace of $F$ on $\Hc^\bullet(\X_{n,d},\mathcal{F}_\mu)$ is actually $\chi_{\mu * x}$. If $n=2$, then the only ambiguity concerns $\chi_{\mu * x}$ when $\mu *x = (1,2)$. In that case, the $q^{n+1+x-s}$-eigenspace can be either $\chi_{\mu * x}$ or $\mathrm{1}_G + \mathrm{St}_G$. But by \cite[Corollary 8.28]{DM3}, the trivial character and the Steinberg character cannot occur in the same cohomology group of $\X_{n,d}$ as soon as the dimension of this variety is non-zero.

\mk

\noindent \textbf{Case (2).} Assume now that $x \notin X$. The $q^{n+1+x-s}$-eigenspace of $F$ on $\mathscr{D}^{(j)}$ is non-zero if and only if $x\in X^{(j)}$ and $x + d \notin X^{(j)}$. Since $x\notin X$, this forces $x =x_j^{(j)}= x_j-1$ and $x_j + d-1 \notin X$. In that case, its character corresponds to  a partition with $\beta$-set $(X \smallsetminus \{x + 1\}) \cup \{x+d\}$ and it appears in degree $2+\pi_d(X^{(j)},x)$ only. On the other hand,  the $q^{n+1+x-s}$-eigenspace of $F$  on $\mathscr{C}$ is non-zero if and only if $x+1 \in X$ and $x+1 + d-1 = x+d \notin X$. By (a), the character of this eigenspace corresponds to a partition with $\beta$-set $(X \smallsetminus \{x + 1\}) \cup \{x+d\}$. Furthermore, it appears in degree  $1+\pi_{d-1}(X,x+1)$ only. Note that in that case we have

\centers{$\begin{array}{r@{\, \ = \, \ }l} \pi_{d-1}(X,x+1) & 2\big(n+x - \#\{y \in X \, | \, y < x+1\}\big) - \#\{y \in X \, | \, x+1 < y < x+d\} \\[4pt] &
 2\big(n+x - \#\{y \in X \, | \, y < x\}\big) - \big( \#\{y \in X \, | \, x < y < x+d\}-1\big) \\[4pt]
 & \pi_d(X,x) +1 \end{array}$}
 
\leftcenters{and}{$\begin{array}[t]{r@{\, \ = \, \ }l} \pi_{d}(X^{(j)},x) & 2\big(n-1+x - \#\{y \in X^{(j)} \, | \, y < x\}\big) - \#\{y \in X^{(j)} \, | \, x < y < x+d\} \\[4pt] &
 2\big(n-1+x - \#\{y \in X \, | \, y < x\}\big) - \big( \#\{y \in X \, | \, x < y < x+d\}-1\big) \\[4pt]
 & \pi_d(X,x) - 1. \end{array}$}

\sk

We deduce that the $q^{n+1+x-s}$-eigenspace of $F$ on $\Hc^\bullet(\X_{n,d},\mathcal{F}_\mu)^{U_I}$ is isotypic and concentrated in two consecutive degrees. However, there are only a few unipotent characters that can have an isotypic Harish-Chandra restriction: they correspond to partitions of the form $(a,a,\ldots,a)$. Among them we can find the Steinberg character $\mathrm{St}_G$ (with $a=1$) and the trivial character $\mathrm{1}_G$ (with $a=n+1$). But by \cite[Corollary 8.28]{DM3} they have respective eigenvalues $1$ and $q^{2n+1-d}$. Let us  write the $\beta$-set of $\mu$ as $X = \{0,1,\ldots,k-1,\mu_1+k,\mu_2+k+1,\ldots,\mu_r+s-1\}$ with $k\geq d$. Since $x \notin X$, one must have $k-1 < x  < \mu_r+s-1$ and hence
\begin{equation} \label{ineq1}d-1 \leq n+ k - s < n+1+ x-s < n+1+\mu_r-1 \leq 2n+1-d.
\end{equation}
\noindent We deduce that the  $q^{n+1+x-s}$-eigenspace of $F$ on $\Hc^\bullet(\X_{n,d},\mathcal{F}_\mu)$ is either zero, or consists of two copies of the character $\chi_\lambda$ in two consecutive degrees, where $\lambda = (a,a,\ldots,a)$ with $1 < a < n+1$. We shall actually prove that it is always zero, but before that we need to study the last case.

\begin{rmk}\label{rmkbetaset}The Harish-Chandra restriction of $\chi_\lambda$ corresponds to the partition $(a-1,a,\ldots,a)$. Therefore if the associated character appears in the cohomology of $\mathscr{C}$ then  the $\beta$-set $(X\smallsetminus\{x+1\}) \cup \{x+d\}$ must correspond to the partition $(a-1,a,\ldots,a)$. This gives a rather strong condition on $X$: we will have either

\centers{$ X \, = \, \{0,1,\ldots,k-1,x+1,b,b+2,b+3, \ldots , \widehat{x+d}, \ldots,b+r\} $}

\noindent with $b+2 \leq x + d \leq b+r$, or

\centers{$ X \, = \, \{0,1,\ldots,k-1,x+1, x+d+2,x+d+3, \ldots,x+d+r\}.$}

\end{rmk}

\noindent \textbf{Case (3).} Finally, assume  that $x \in X$ and $x +d \in X$. The $q^{n+1+x-s}$-eigenspace of $F$ on $\mathscr{D}^{(j)}$ is non-zero if and only if $x\in X^{(j)}$ and $x + d \notin X^{(j)}$. Since $x + d \in X$, this forces $x + d =x_j$ (and therefore $x_j-1 \notin X$). In that case, the character of the eigenspace corresponds to  a partition with $\beta$-set $(X^{(j)} \smallsetminus \{x\}) \cup \{x+d\} = (X \smallsetminus \{x\}) \cup \{x+d-1\}$. On $\mathscr{C}$, the Frobenius has a non-zero $q^{n+1+x-s}$-eigenspace if and only if $x + d -1 \notin X$ and its character is again associated to the $\beta$-set  $(X \smallsetminus \{x\}) \cup \{x+d-1\}$. This ensures that the $q^{n+1+x-s}$-eigenspace of $F$ on $\Hc^\bullet(\X_{n,d},\mathcal{F}_\mu)^{U_I}$ is isotypic. Using $x+d-1$ instead of $x$ in the inequalities \ref{ineq1} yields $0 < n+1+x-s < 2n+2-2d$ and therefore the previous argument applies.  We deduce that the  $q^{n+1+x-s}$-eigenspace of $F$ on $\Hc^\bullet(\X_{n,d},\mathcal{F}_\mu)$ is again either zero or consists of two copies of the character $\chi_\lambda$ in two consecutive degrees, namely $\pi_d(X,x)-1$ and $\pi_d(X,x)$, where $\lambda = (a,a,\ldots,a)$ and $1 < a < n+1$.

\mk

To conclude, we need to prove that the $q^{n+1+x-s}$-eigenspaces of $F$ are actually zero whenever $x \notin X$ or $x+d \in X$. Let us first summarize what we have proven so far:

\begin{itemize}

\item[(1)] if $x \in X$ and $x +d \notin X$ then the $q^{n+1+x-s}$-eigenspace of $F$ on $\Hc^\bullet(\X_{n,d},\mathcal{F}_\mu)$ is $\chi_{\mu * x}$ and it appears in degree $\pi_d(X,x)$ only; 

\item[(2)] if $x \notin X$, the $q^{n+1+x-s}$-eigenspace of $F$ is zero unless $x+1 \in X$ and $x+d \notin X$. In that case, it may consist of two copies of $\chi_\lambda$, one in degree $\pi_d(X,x)+1$ and one in degree $\pi_d(X,x)+2$, where $\lambda = (a,a,\ldots,a)$ with $1 < a <n+1$. 
Moreover, the $\beta$-set  $(X\smallsetminus \{x+1\}) \cup \{x+d\}$ must correspond to  the partition $(a-1,a,\ldots,a)$ (see Remark \ref{rmkbetaset});

\item[(3)]  if $x \in X$ and $x+d\in X$, the $q^{n+1+x-s}$-eigenspace of $F$ is zero unless $x+d-1 \notin X$. In that case, it can only be $\chi_\lambda$-isotypic with $\lambda = (a,a,\ldots,a)$ and $1<a<n+1$. Moreover, it is non-zero in degrees $\pi_d(X,x)-1$ and $\pi_d(X,x)$ only, and $(X\smallsetminus \{x\}) \cup \{x+d-1\}$ must be a $\beta$-set of the partition $(a-1,a,\ldots,a)$.

\end{itemize}

\noindent Now, if we assume that Conjecture \ref{conjA} holds for the variety $\X_{n,d+1}$, then we can use the distinguished triangle

\centers{$ \Rgc(\mathbb{G}_m \times \mathrm{X}_{n,d}, \qlb \otimes \mathcal{F}_\mu) \longrightarrow \Rgc( \X_{n+1,d+1},\mathcal{F}_\mu)^{U_I} \longrightarrow \Rgc\big(\mathrm{X}_{n,d+1},\displaystyle \bigoplus \mathcal{F}_{\mu^{(j)}}\big)[-2] (1) \rightsquigarrow$}

\noindent from Theorem \ref{restrA} to prove that the eigenspaces of $F$ on $\Hc^{\bullet}(\X_{n,d},\mathcal{F}_\mu)$ in cases (2) and (3) are indeed zero. 

\sk

Assume that $x\notin X$ and that there exists $1 < a < n+1$ such that the character $\chi_\lambda= \chi_{(a,a,\ldots,a)}$ appears twice in the $q^{n+1+x-s}$-eigenspace of $F$ on the cohomology of $\X_{n,d}$ $-$  that is in degrees $\pi_{d}(X,x)+1$ and $\pi_d(X,x)+2$.  Then, 

\begin{itemize}

\item if $x-1 \notin X$, the $q^{n+1+x-s}$-eigenspace of $F$ on $\Hc^{\bullet}(\mathbb{G}_m \times \mathrm{X}_{n,d}, \qlb \otimes \mathcal{F}_\mu)$ is $\chi_\lambda$-isotypic by (2) (we have $x-1 +1 \notin X$). Moreover, the $q^{n+1+x-s}$-eigenspace of $F$ on $\Hc^\bullet\big(\mathrm{X}_{n,d+1},\displaystyle \bigoplus \mathcal{F}_{\mu^{(j)}}\big)[-2] (1) $ is zero since Conjecture \ref{conjA} holds for the variety $\X_{n,d+1}$. We deduce that the eigenspace on $ \Hc^{\bullet}(\X_{n+1,d+1},\mathcal{F}_\mu)^{U_I}$ is $\chi_\lambda$-isotypic, which is impossible since no unipotent character can have $\chi_\lambda$ as a Harish-Chandra restriction when $1 < a < n+1$.

\item if $x-1 \in X$, then since $x-1 + d+1 \notin X$, the $q^{n+1+x-s}$-eigenspace of $F$ on $\Hc^{\bullet}(\mathbb{G}_m \times \mathrm{X}_{n,d}, \qlb \otimes \mathcal{F}_\mu)$ and $\Hc^\bullet\big(\mathrm{X}_{n,d+1},\displaystyle \bigoplus \mathcal{F}_{\mu^{(j)}}\big)[-2] (1) $ can be determined as in case (1).  It corresponds to the Harish-Chandra restriction of the partition $\mu *(x-1)$ obtained from $\mu$ by adding a $(d+1)$-hook from $x-1$. Furthermore, they will appear in degree $\pi_{d+1}(X,x-1)$ only, which is equal to

\centers{$\begin{array}{r@{\, \ = \, \ }l} \pi_{d+1}(X,x-1) & 2\big(n+x - \#\{y \in X \, | \, y < x-1\}\big) - \#\{y \in X \, | \, x-1 < y < x+d\} \\[4pt] &
 2\big(n+x +1- \#\{y \in X \, | \, y < x\}\big) - \#\{y \in X \, | \, x < y < x+d\} \\[4pt]
 & \pi_d(X,x) +2. \end{array}$}

\noindent To these characters we have to add the contribution of $\chi_\lambda$ and possibly of an other character $\chi_{\lambda'}$ corresponding to $\lambda = (a',a',\ldots,a')$ (when the case (3) applies to $x-1$).  Now, we claim that neither $\chi_\lambda$ nor $\chi_\lambda'$ can appear in the $q^{n+1+x-s}$-eigenspace of $F$ on $\Hc^\bullet\big(\mathrm{X}_{n,d+1},\displaystyle \bigoplus \mathcal{F}_{\mu^{(j)}}\big)[-2] (1) $. Indeed the assumptions on $x$  force $X$ (see Remark \ref{rmkbetaset}) to be either

\centers{$ X \, = \, \{0,1,\ldots,k-1,k+1,b,b+2,b+3, \ldots , \widehat{k+d}, \ldots,b+r\} $}

\noindent with $b+2 \leq k+d \leq b+r$, or

\centers{$ X \, = \, \{0,1,\ldots,k-1,k+1, k+d+2,k+d+3, \ldots,k+d+r\}.$}

\noindent Therefore a $\beta$-set corresponding to the partition $\mu * (x-1)$ of $n+2$ is either

\centers{$\{0,1,\ldots,k-2,k+1,b,b+2,b+3, \ldots,b+r\} $}

\noindent with $b+2 \leq k + d \leq b+r$, or

\centers{$ \{0,1,\ldots,k-2,k+1, k+d, k+d+2,k+d+3, \ldots,k+d+r\}.$}

\noindent We deduce that the restriction of $\mu * (x-1)$ will never produce $\lambda$ or $\lambda'$ unless $r=2$, $b = k+2$ and $d= 4$ in the first case, or $r=2$ and $d=2$ in the second case. In these very specific cases, we have either $X = \{0,\ldots,k-1,k+1,k+2\}$, which corresponds to the partition $\mu = (1,1)$ or $X = \{0,\ldots,k-1,k+1,k+4\}$, which corresponds to $\mu = (1,3)$. In this situation, we get $\lambda = (3,3)$ and $\mu * (x-1) = (2,2,3)$. But $(3,3)$ cannot be obtain by restricting $\mu * (x-1)$. This proves that the $q^{n+1+x-s}$-eigenspace of $F$ on $\Hc^{\pi_d(X,x)+2}(\mathrm{X}_{n+1,d+1},\mathcal{F}_\mu)^{U_I}$ is just $\chi_\lambda$ (plus possibly $\chi_{\lambda'}$), which is impossible by the properties of the Harish-Chandra restriction.
\end{itemize}

The same argument can be adapted to deal with the case (3), if we rather look at the $q^{n+2+x-s}$-eigenspace and  distinguish whether $x+1+d$ is an element of $X$ or not. The details are left to the reader.
\end{proof}

\begin{center} \subsection*{Acknowledgments} \end{center}

I would like to thank David Craven for introducing me to the results in \cite{Cra} and  for many stimulating discussions around the  formula he discovered. I would also like to thank the Hong Kong University of Science and Technology, and especially Xuhua He, for providing me excellent research environment during the early stage of this work.

\bk

\bibliographystyle{abbrv}
\bibliography{typeAchar0}

\end{document}